\documentclass{amsart}

\usepackage{graphicx}
\usepackage{mathptmx}      % use Times fonts if available on your TeX system
\usepackage{amsmath}
\usepackage{amsthm}
\usepackage{amssymb}
\usepackage{enumerate}
\usepackage{fullpage}

\newcommand{\R}{\mathbb{R}}
\newcommand{\N}{\mathbb{N}}
\newcommand{\cad}{c\`{a}dl\`{a}g}
\DeclareMathOperator{\Disc}{Disc}

\newtheorem{lemma}{Lemma}
\newtheorem{theorem}{Theorem}
\newtheorem{proposition}{Proposition}

\theoremstyle{definition}
\newtheorem{example}{Example}
\newtheorem{rem}{Remark}

\newtheoremstyle{myquest}{\topsep}{\topsep}%
     {}%         Body font
     {}%         Indent amount (empty = no indent, \parindent = para indent)
     {\bfseries}% Thm head font
     {:}%        Punctuation after thm head
     {.5em}%     Space after thm head (\newline = linebreak)
     {}%         Thm head spec
\theoremstyle{myquest}
\newtheorem*{quest}{Q}

\begin{document}

\title{Weak convergence of stochastic integrals driven by \\continuous-time random walks}

\author{Meredith N. Burr}

\address{Mathematics and Computer Science Department, Rhode Island College,\newline
              \indent 600 Mount Pleasant Avenue, Providence, RI, 02908, USA}
\email{mburr@ric.edu}
\thanks{This paper is based on the author's dissertation, Continuous-time random walks, their scaling limits, and connections with stochastic integration, at Tufts University under the direction of Marjorie Hahn.}

\date{\today}

\begin{abstract}
Brownian motion is a well-known model for normal diffusion, but not all physical phenomena behave according to a Brownian motion.  Many phenomena exhibit irregular diffusive behavior, called anomalous diffusion.  Examples of anomalous diffusion have been observed in physics, hydrology, biology, and finance, among many other fields.  Continuous-time random walks (CTRWs), introduced by Montroll and Weiss, serve as models for anomalous diffusion.  CTRWs generalize the usual random walk model by allowing random waiting times between successive random jumps.  Under certain conditions on the jumps and waiting times, scaled CTRWs can be shown to converge in distribution to a limit process $M(t)$ in the \cad\ space $\mathbb{D}[0,\infty)$ with the Skorohod $J_1$ or $M_1$ topology.  An interesting question is whether stochastic integrals driven by the scaled CTRWs $X^n(t)$ converge in distribution to a stochastic integral driven by the CTRW limit process $M(t)$.  We prove weak convergence of the stochastic integrals driven by CTRWs for certain classes of CTRWs, when the CTRW limit process is an $\alpha$-stable L\'{e}vy motion and when the CTRW limit process is a time-changed Brownian motion.\\

\end{abstract}

\maketitle

\section{Introduction}
\label{intro}
Continuous-time random walks (CTRWs) were developed by Montroll and Weiss \cite{Montroll+Weiss} to study random walks on a lattice.  CTRWs generalize the usual random walk model by introducing random waiting times between jumps.  The addition of random waiting times allows CTRWs to capture irregular diffusive behavior, making CTRWs good models for anomalous diffusion.  Unlike normal diffusion, anomalous diffusion is characterized by a non-linear relationship between the mean squared displacement of a particle and time.  Anomalous diffusion has been observed in hydrology, where contaminants often travel more slowly in ground water due to sticking or trapping, in biology, where proteins diffuse more slowly across cell membranes, and in the behavior of many financial markets where heavy-tailed price jumps occur \cite{Klafter+Sokolov,MS:2008}.  Additional examples of anomalous diffusion can be found in physics, hydrology, biology, and finance, among many other fields \cite{Klafter+Sokolov,MS:2008,Metzler+Klafter,Scalas2006,SM:1975}.  Because of the existence of so many examples of anomalous diffusion, it is important to be able to model and to understand this type of behavior.

Physicists often use the CTRW model to derive a partial differential equation which describes the anomalous diffusion.  The differential operators in the normal diffusion equation $\frac{\partial p(x,t)}{\partial t} = D \frac{\partial^2 p(x,t)}{\partial x^2}$ may be replaced by fractional derivatives in time and space or even pseudo-differential operators in time and in space \cite{MS:2004,MS:2008}.  The density of the ``long-time scaling limit'' of the CTRW solves the space-time-fractional diffusion equation derived from the CTRW model \cite{MS:2004,Metzler+Klafter}.  Moreover, the long-time scaling limit of the CTRW is actually the weak limit of a sequence of normalized CTRWs.   Meerschaert and Scheffler derive the scaling limit of the CTRW under certain conditions and identify the CTRW limit process as a time-changed L\'{e}vy motion (see \cite{MS:2004} for details).  Studying the CTRW limit process and the processes it drives is important, since the dynamics of the limit process correspond to its anomalous diffusion equation.

Many phenomena are modeled by a stochastic differential equation driven by the CTRW limit process (for example, see \cite{Chakraborty,MWW}).  For applications, it is also important to consider how the stochastic integrals driven by the CTRW limit process can be approximated.  Let $\mathbb{D}_{\R^k}$ denote the \cad\ function space of $\R^k$-valued functions,  $J_1$ the standard Skorohod topology on $\mathbb{D}$ \cite{Billingsley}, and $\Rightarrow$ convergence in distribution.  We study the following question:
\begin{quest}
If $X^n(t)$ is a scaled CTRW and if $(H^n,X^n)\Rightarrow (H,M)$ in $(\mathbb{D}_{\R^2}, J_1)$ where $M(t)$ is the CTRW limit process, then under what conditions does
\begin{equation}
\label{eqn:question}
\left\{\int_0^t H^n(s-) \text{ }dX^n(s)\right\}_{t\geq 0} \Longrightarrow \left\{\int_0^t H(s-) \text{ }dM(s)\right\}_{t\geq 0} \text{ in } (\mathbb{D}_\R,J_1) \text{ as } n\rightarrow \infty?
\tag{$\ast$}
\end{equation}
\end{quest}
A paper of Germano \textit{et al.} conjectures that for the usual compound Poisson process $X(t)$, the stochastic integrals of the scaled process $X^n(t)$ driven by $X^n(t)$ converge weakly to the stochastic integral of Brownian motion driven by Brownian motion \cite{Scalas2010}, i.e. $\{\int_0^t X^n(s-)dX^n(s)\}_{t\geq 0} \Rightarrow \{\int_0^t B(s)dB(s)\}_{t\geq 0}$.  In this paper, we verify this conjecture as a special case of a more general result (see Example \ref{ex:Germano_conj}).  We show that (\ref{eqn:question}) holds in the case that the limit process is an $\alpha$-stable L\'{e}vy motion (Theorem \ref{thm1}), and in the case in which the limit process is a time-changed Brownian motion (Theorem \ref{thm2}).  In order to prove these results we rely on Kurtz and Protter's theorem on weak convergence of stochastic integrals \cite{KPDurham}.  The key condition is that the integrators $(X^n)$ be \textit{uniformly tight}.

In Section \ref{sec:background}, we recall the required facts on continuous-time random walks, scaling limits, and weak convergence of stochastic integrals.  Our main results are contained in Section \ref{sec:Results} as well as the example proving the conjecture of Germano \textit{et al.}  Section \ref{sec:App} focuses on some applications and examples of the main results.
\section{Background}
\label{sec:background}
\subsection{Continuous-time random walks}
\label{subsec:CTRWs}
Under the random walk model, random displacements occur at regular deterministic intervals.  The continuous-time random walk (CTRW) model allows random waiting times between successive random displacements, or jumps.  Let $(\xi_i)_{i\in \N}$ be a sequence of independent and identically distributed (i.i.d.) $\R$-valued random variables representing the particle jumps and $(\tau_i)_{i\in \N}$ be a sequence of i.i.d. nonnegative random variables representing the waiting times between successive jumps.  Throughout, the focus is on uncoupled CTRWs in which the sequences of jumps $(\xi_i)$ and waiting times $(\tau_i)$ are independent.  Set $S(0)=0$ and $T(0)=0$ and let $S(n) = \xi_1 + \hdots + \xi_n$ be the position of the particle after $n$ jumps and $T(n) = \tau_1 + \hdots + \tau_n$ be the time of the $n$th jump.  For $t\geq 0$, define
\begin{equation}
N(t) := \max\{n\geq 0: T(n) \leq t\}
\end{equation}
to be the number of jumps by time $t$.  The position of the particle at time $t$ can then be written as the sum of the jumps up to time $t$:
\begin{equation}
X(t) = S(N(t)) = \sum_{i=1}^{N(t)} \xi_i.
\end{equation}
The stochastic process $\{X(t)\}_{t\geq 0}$ is called a \textit{continuous-time random walk (CTRW)}.

\begin{figure}[htb]
\begin{center}
\begin{picture}(170,150)
\input{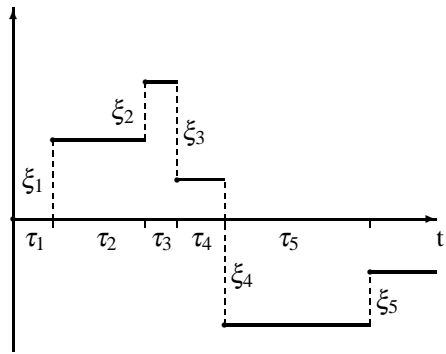}
\end{picture}
\caption{A continuous-time random walk.}
\end{center}
\end{figure}

Since the waiting times $(\tau_i)$ are i.i.d., $N(t)$ is a renewal counting process and the CTRW is a compound renewal process, or renewal-reward process (as seen in queueing theory).  A CTRW may also be called a point process with reward.  The counting process $N(t)$ and the the time of the $n$th jump $T(n)$ have a particularly nice inverse relationship.  For any $n\geq 0$ and any $t\geq 0$,
\begin{equation}
\{N(t)\geq n\}=\{T(n)\leq t\}.
\label{eqn:relationship}
\end{equation}
Equation (\ref{eqn:relationship}) holds because the number of jumps by time $t$ is at least $n$ if and only if the $n$th jump occurs at or before time $t$.

Every CTRW is a semi-Markov process and is Markovian if and only if the distribution of the waiting times is exponential \cite{Scalas2009,HPS}.  Germano \textit{et al.} \cite{Scalas2009} prove that a CTRW is also a martingale under certain conditions.  Since we require this result in Section \ref{sec:Results}, we include an alternate (and more direct) proof here.

\begin{lemma} %cite Scalas stoch calc paper?
Let $X(t) = \sum_{i=1}^{N(t)} \xi_i$ be a CTRW.  Assume $E(\xi_1)=0$, $E\big|\xi_1\big|<\infty$, and $E[N(t)] < \infty$ for each $t$.  Then $X(t)$ is a martingale.
\label{lem:martingale}
\end{lemma}
\begin{proof}
First observe that $X(t) \in L^1$ for all $t\geq 0$ by Wald's Equality,
\begin{equation*}
E\Bigg| \sum_{i=1}^{N(t)} \xi_i \Bigg| \leq E \sum_{i=1}^{N(t)} |\xi_i|
= E|\xi_1|E[N(t)] < \infty.
\end{equation*}

Now let $\mathcal{F}_t = \sigma\{X(s), 0\leq s\leq t\}$, the natural filtration of $X$.  Then $X(t)$ is adapted to $\mathcal{F}_t$, and for $s\leq t$,
\begin{align*}
E[X(t)| \mathcal{F}_s] &= E[X(t) - X(s)| \mathcal{F}_s] + E[X(s)|\mathcal{F}_s]\notag\\
&= E\Bigg[\sum_{i=N(s) + 1}^{N(t)} \xi_i \Big|\mathcal{F}_s\Bigg] + X(s)\notag\\
&= E\Bigg[\sum_{i=N(s) + 1}^{N(t)} \xi_i\Bigg] + X(s) \quad \text{ since $\xi_i$ are i.i.d.}\notag\\
&= E[\xi_1](E[N(t)] - E[N(s)]) + X(s) \quad \text{ by Wald's Equality} \notag\\
&= X(s) \text{ since $E[\xi_1] = 0$}.
\end{align*}

%\vspace{-.25in}
%\qed
\end{proof}

If $X(t)$ is a CTRW satisfying the conditions in Lemma \ref{lem:martingale} and $H(t)$ a \cad\ adapted process, then $X(t)$ is a martingale and we can understand $\int_0^t H(s-)dX_s$ in the sense of Protter \cite{Protter}.  But we can also make a simpler definition.  Germano \textit{et al.} \cite{Scalas2009} note that this integral can be defined as a sum, since the CTRW is a step function, i.e.,
\begin{equation}
\int_0^t H(s-) dX(s) = \sum_{i=1}^{N(t)} H(T_i-)\xi_i,
\end{equation}
where the integrand is evaluated at the jump times of $X$.  The right-hand side can be considered as a coupled CTRW, since the new jumps $H(T_i-)\xi_i$ clearly depend on $N(t)$.

\subsection{Scaling limits}
\label{subsec:scaling}
Here we recall the known scaling limit theorems for CTRWs.  By the scaling limit of a CTRW, we mean the limit process resulting from appropriate scaling in time and space according to a functional central limit theorem (FCLT).  The limit behavior of the CTRW depends on the distribution of the jumps and the waiting times.  If the waiting times have finite mean, the CTRW behaves like a random walk in the limit.  So, by Donsker's Theorem, if the waiting times have finite mean and the jumps have finite variance then the scaled CTRW converges in distribution to a Brownian motion.  If the waiting times have finite mean and the jumps are in the domain of attraction (DOA) of an $\alpha$-stable random variable, with $\alpha \in (0,2)$, then the appropriately scaled CTRW converges in distribution to an $\alpha$-stable L\'{e}vy motion {\cite[Theorem 4.5.3]{Whitt}}.

However, if the waiting times have an infinite mean, the CTRW limit behavior is more complex.  For this case, Meerschaert and Scheffler prove a FCLT which identifies the limit process as a composition of an $\alpha$-stable L\'{e}vy motion $A(t)$ and the inverse of a $\beta$-stable subordinator $E(t)$, where $\alpha \in (0,2]$ and $\beta \in (0,1)$ (see Theorem 4.2 in \cite{MS:2004} for details).  Their proof uses a continuous-mapping approach and the convergence holds in the $M_1$-topology and not in the stronger $J_1$-topology.  By examining the proof, we note that their convergence result only holds in the $M_1$-topology since the composition map is $M_1$ but not $J_1$-continuous at $(A,E)$.  However, in the case that $\alpha = 2$, we can prove that the the convergence holds in the $J_1$-topology.  This is important since the theorems on weak convergence of stochastic integrals are stated in the $J_1$-topology.  We provide a proof (in $\mathbb{R}$) for the case $\alpha=2$ here following the proof of Meerschaert and Scheffler.  The resulting limit process, a time-changed Brownian motion $B(E_t)$, is of particular interest since it is a good model for subdiffusion.

\begin{theorem}
Let $(\tau_i)_{i\in \N}$ be i.i.d. nonnegative random variables in the strict domain of attraction of a $\beta$-stable, $\beta \in (0,1)$.  Let $(\xi_i)_{i\in \N}$ be i.i.d. $\R$-valued random variables in the strict domain of attraction of a normal (i.e., an $\alpha$-stable with $\alpha=2$).  Let $\displaystyle X(t) = \sum\limits_{i=1}^{N_{t}} \xi_i$.
Then there exists a regularly varying function $\tilde{c}$ of index $\beta/2$ with $\tilde{c}(n)=c(\tilde{b}(n))$, $\tilde{b}$ regularly varying of index $\beta$, and $c(n) = \sqrt{n}L(n)$ for some slowly varying function $L$ such that, as $n\rightarrow \infty$,
\begin{equation*}
\bigg\{\frac{X(nt)}{\tilde{c}(n)}\bigg\}_{t\geq 0} \Rightarrow \{B(E(t))\}_{t\geq 0} \quad \text{ in } (\mathbb{D}_{\R},J_1).
\end{equation*}
\label{thm:B(E)}
\end{theorem}
\begin{proof}
We follow the continuous-mapping approach used in \cite{MS:2004}, substituting the $J_1$-continuity of composition theorem instead of the $M_1$.  Let $\displaystyle B^n(t) = \frac{S(\tilde{b}(n)t)}{\tilde{c}(n)}$ and $\displaystyle E^n(t) = \frac{N(nt)}{\tilde{b}(n)}$.  By the independence of the jumps $(\xi_i)$ and the waiting times $(\tau_i)$, Theorem 4.1 and Corollary 3.4 in \cite{MS:2004} imply that as $n\rightarrow \infty$,
\begin{equation*}
\{(B^n(t),E^n(t))\}_{t\geq 0} \Rightarrow \{B(t),E(t)\}_{t\geq 0} \quad \text{ in } (\mathbb{D}_\R \times \mathbb{D}_{\R_+}, J_1).
\end{equation*}
Consider the continuous-mapping theorem {\cite[Theorem 3.4.3]{Whitt}} where the map is composition, $\circ:(\mathbb{D}_\R \times \mathbb{D}_{\R_{+}},J_1) \rightarrow (\mathbb{D}_\R,J_1)$.  Let $\mathbb{C}$ be the space of continuous functions, $\mathbb{D}_{\uparrow}$ the space of nondecreasing \cad\ functions, and $\mathbb{C}_{\upuparrows}$ the space of strictly increasing continuous functions.  By the $J_1$-continuity of composition theorem {\cite[Theorem 13.2.2]{Whitt}}, the set of discontinuity points of composition under $J_1$ is $\Disc(\circ) = ((\mathbb{C} \times \mathbb{D}_{\uparrow}) \cup (\mathbb{D} \times \mathbb{C}_{\upuparrows}))^c$.  Since $(B(t),E(t))\in \mathbb{C} \times \mathbb{D}_{\uparrow}$, $P((B,E) \in \Disc(\circ)) = 0$.  Since $\circ$ is measurable by \cite[p. 145]{Billingsley}, the continuous-mapping theorem implies that $B^n \circ E^n \Rightarrow B\circ E$ in the $J_1$ topology.  Consequently, as $n\rightarrow \infty$,
\begin{equation*}
\bigg\{\frac{X(nt)}{\tilde{c}(n)}\bigg\}_{t\geq 0} \Rightarrow \{B(E(t))\}_{t\geq 0} \quad \text { in }
(\mathbb{D}_{\R},J_1).
\end{equation*}

%\vspace{-.25in}
%\qed
\end{proof}

\subsection{Weak convergence of stochastic integrals}
\label{subsec:ConvThms}
We now review the conditions under which a sequence of stochastic integrals converges in distribution.  Let $\Theta^n = (\Omega^n, \mathcal{F}^n, (\mathcal{F}^n_t)_{t\geq 0}, P^n)$ be a sequence of filtered probability spaces satisfying the usual hypotheses: (1) $\mathcal{F}^n_0$ contains all the $P^n$-null sets of $\mathcal{F}^n$ and (2) the filtration $(\mathcal{F}^n_t)$ is right-continuous, i.e., $\mathcal{F}^n_t = \cap_{\varepsilon >0} \mathcal{F}^n_{t+\varepsilon}$.  Assume $(H^n)$ is a sequence of \cad\ $\mathcal{F}_t^n$-adapted processes, $X^n$ a sequence of \cad\ $\mathcal{F}_t^n$-semimartingales, and $X^n \Rightarrow X$.  Kurtz and Protter call $(X^n)$ \textit{good} if for any sequence $(H^n)$ in $\mathbb{D}_{\R}$ defined on $\Theta^n$ such that $(H^n,X^n)\Rightarrow (H,X)$ in $(\mathbb{D}_{\R^2},J_1)$, then there exists a filtration $(\mathcal{F}_t)$ such that $X$ is an $(\mathcal{F}_t)$-semimartingale and $\int H^n_{s-} \text{ }dX^n_s \Rightarrow \int H_{s-}\text{ }dX_s$ in $(\mathbb{D}_{\R},J_1)$ {\cite[Definition 13]{KP1995}}.

Necessary and sufficient conditions for goodness are discussed in \cite{KPDurham,KP1995}.  First, it is necessary to assume $(H^n,X^n)$ converges jointly to $(H,X)$ in $\mathbb{D}_{\mathbb{R}^2}$ and not in $\mathbb{D}_\mathbb{R} \times \mathbb{D}_\mathbb{R}$.  If this joint convergence holds, then the key condition is that the integrators $(X^n)$ be \textit{uniformly tight} (UT) or, equivalently, have uniformly controlled variations (UCV) \cite{KP1995}.  In fact, if $(H^n,X^n) \Rightarrow (H,X)$ in $(D_{\mathbb{R}^2},J_1)$, then $(X^n)$ is good if and only if $(X^n)$ satisfies UT (equivalently UCV) {\cite[Theorem 32, 34]{KP1995}}.  We use the UT condition throughout this paper.  The definition of uniform tightness was first introduced by Jakubowski, M\'{e}min, and Pag\`{e}s \cite{JMP}.  A sequence of semimartingales $(X^n)$ is said to be \textit{uniformly tight} (UT), if for each $t>0$, the set $\{\int_0^t H^n_{s-}dX^n_s, H^n \in \mathbf{S}^n, |H^n|\leq 1, n\geq 1\}$ is stochastically bounded (uniformly in $n$), where $\mathbf{S}^n$ denotes the collection of simple predictable processes on $\Theta^n$.  Since we desire the weak convergence of the stochastic integrals $\int H^n_{s-} \text{ }dX^n_s \Rightarrow \int H_{s-}\text{ }dX_s$, it makes sense to require tightness of the laws of $\int H^n_{s-} \text{ }dX^n_s$.  However, it is not easy to verify the UT condition directly.  We rely on the following proposition, which gives a more accessible condition for UT to hold.  Let $|\Delta X^n(t)| = |X^n(t) - X^n(t-)|$ be the size of a jump at time $t$.

\begin{proposition}\cite{JMP,KP1995}
If $(X^n)_{n\geq 1}$ is a sequence of local martingales and if for each $t<\infty$,
\begin{equation}
\sup_n E^n\left[\sup_{s\leq t}|\Delta X^n(s)|\right] < \infty,
\label{eqn:sup}
\end{equation}
then $(X^n)$ satisfies UT.
\label{prop:UT}
\end{proposition}

\section{Results}
\label{sec:Results}
Our first result concerns the weak convergence of a stochastic integral driven by a CTRW with deterministic waiting times to a stochastic integral driven by the CTRW limit process $A(t)$, an $\alpha$-stable L\'{e}vy motion.  Let $S_{\alpha} = S_{\alpha}(1,\gamma,0)$ be the family of $\alpha$-stable distributions with skewness parameter $\gamma \in [-1,1]$.  Our proof of the following theorem requires that the jumps $\xi_i$ have a finite first moment, so we assume $\alpha \in (1,2]$.  This assumption is sufficient for most applications involving superdiffusion.

\begin{theorem}
Let $(\xi_i)_{i\in \N}$ be i.i.d. mean 0 random variables.  Assume $\xi_1$  belongs to the domain of attraction of an $\alpha$-stable random variable $S_{\alpha}$, $\alpha \in (1,2]$.  Define
\begin{equation}\label{eq:stable CTRW}
X^n(t) = \frac{S(nt)}{a(n)} = \sum_{i=1}^{\lfloor nt \rfloor} \frac{\xi_i}{a(n)},
\end{equation}
where $a(n) = n^{1/\alpha}L(n)$ with $L$ a slowly varying function.  If $(H^n, X^n) \Rightarrow (H,A)$ in $(\mathbb{D}_{\R^2},J_1)$, then there exists a filtration $(\mathcal{F}_{t})$ such that $A$ is an $(\mathcal{F}_t)$-semimartingale, $H$ is an $(\mathcal{F}_t)$-adapted \cad\ process, and as $n\rightarrow \infty$
\begin{equation}
\left\{\int_0^t H^n(s-) \text{ }dX^n(s)\right\}_{t\geq 0} \Longrightarrow \left\{\int_0^t H(s-) \text{ }dA(s)\right\}_{t\geq 0} \text{ in } (\mathbb{D}_{\R},J_1).
\end{equation}
\label{thm1}
\end{theorem}

Our goal is to show Theorem \ref{thm1} follows from Proposition \ref{prop:UT} and Theorem 34 in \cite{KP1995} once the conditions of Proposition \ref{prop:UT} are verified.

\begin{proof}
We first show that $(X^n(t))_{n\geq 1}$ is a sequence of martingales.  Note that the jumps $\xi_i$ have mean 0 and finite first moment since $\alpha > 1$.  By the same techniques used in Lemma \ref{lem:martingale}, we see that $X^n(t)$ is an $\mathcal{F}_t^n$-martingale (and therefore a local martingale) for each $n$, where $\mathcal{F}_t^n$ is the natural filtration of $X^n$.  Therefore, the first hypothesis of Proposition \ref{prop:UT} holds.

To verify the second hypothesis, we need to show that $\sup_n E^n[\sup_{s\leq t}|\Delta X^n(s)|]<\infty$ for each $t<\infty$.  Fix $t < \infty$.  Observe that, since $X^n(s)$ is a c\`{a}dl\`{a}g step process, it has jumps of size $|\frac{\xi_k}{a(n)}|$ at times $k/n$, $k\in \N$.  Additionally, since $X^n(t)$ has finitely many jumps by time $t$, we can replace the $\sup$ up to time $t$ by a $\max$ up to $\lfloor nt \rfloor$.  So we need to find for each $t$ a uniform bound in $n$ for $E^n\Big[\max_{1\leq k\leq \lfloor nt \rfloor}\big|\frac{\xi_k}{a(n)}\big|\Big]$.
Let $(\epsilon_k)$ be a sequence of i.i.d. Rademacher random variables, independent of $(\xi_k)$.  Then $|\xi_k|=|\epsilon_k\xi_k|$ and $\epsilon_k\xi_k$ are i.i.d. in $DOA(S_\alpha)$.  Let $\displaystyle \widetilde{S}(nt)=\sum_{k=1}^{\lfloor nt \rfloor} \epsilon_k\xi_k$, the symmetrized sum.  By L\'{e}vy's Inequality,
\begin{align*}
E^n\Bigg[\max_{1\leq k\leq \lfloor nt \rfloor}\bigg|\frac{\xi_k}{a(n)}\bigg|\Bigg]
&=\int_0^\infty P^n\Bigg\{\max_{1\leq k\leq \lfloor nt \rfloor}\bigg|\frac{\epsilon_k\xi_k}{a(n)}\bigg| > x \Bigg\} \text{ }dx\\
&\leq \int_0^\infty 2P^n\Bigg\{\bigg|\sum_{k=1}^{\lfloor nt \rfloor}\frac{\epsilon_k\xi_k}{a(n)}\bigg| > x\Bigg\} \text{ }dx
= 2E^n\Bigg|\frac{\widetilde{S}(nt)}{a(n)} \Bigg|.
\end{align*}

\noindent Note that for each $n$,
\begin{equation*}
2E^n\bigg|\frac{\widetilde{S}(nt)}{a(n)} \bigg| \leq 2\sum_{i=1}^{\lfloor nt \rfloor}
E^n\bigg|\frac{\epsilon_i\xi_i}{a(n)}\bigg| = 2\lfloor nt \rfloor E^n\bigg|\frac{\xi_1}{a(n)}\bigg| < \infty.
\end{equation*}
$\xi_1$ has a finite first moment since it is in the domain of attraction of $S_\alpha$, $\alpha \in (1,2]$.  So, for $n$ small, the expectation is bounded.  It remains to show the expectation is bounded in the case of large $n$.  Observe that
\begin{equation*}
2E^n\bigg|\frac{\widetilde{S}(nt)}{a(n)} \bigg| =
2\bigg(\frac{\lfloor nt \rfloor}{n}\bigg)^{1/\alpha}\frac{L(\lfloor nt \rfloor)}{L(n)}
E^n\bigg|\frac{\widetilde{S}(nt)}{a(\lfloor nt\rfloor)} \bigg|.
\end{equation*}

\noindent Since $\epsilon_k\xi_k$ are in $DOA(S_\alpha)$ for $\alpha \in (1,2]$, then by \cite[Exercise 9, p. 91]{AG} the following first moments converge: %Using Proposition \ref{prop:AG p.91} with $b(n)=0$ and $p = 1$,
\begin{equation}
E^n\Bigg|\frac{\widetilde{S}(nt)}{a(\lfloor nt \rfloor)} \Bigg| \rightarrow E|S_{\alpha}| \text{ as } n\rightarrow \infty.
\label{eqn:conv_moments}
\end{equation}

By applying Equation (\ref{eqn:conv_moments}) together with a corollary by Feller in \cite[p. 274]{Feller} on slowly varying functions, there exists some $n_0 \in \N$ such that for all $n > n_0$,
\begin{equation*}
2E^n\bigg|\frac{\widetilde{S}(nt)}{a(n)} \bigg| \leq 2ct^{1/\alpha}(E|S_\alpha| + 1) < \infty ,
\end{equation*}
for $c>1$, where $S_\alpha$ has a finite first moment since $\alpha\in (1,2]$.

So the desired inequality holds:
\begin{align*}
\sup_n E^n\bigg\{\sup_{s\leq t}|\Delta X^n(s)|\bigg\}\notag
&= \max\bigg(\max_{1\leq n\leq n_0} E^n\Big\{\sup_{s\leq t}|\Delta X^n(s)|\Big\},
\sup_{n>n_0} E^n\Big\{\sup_{s\leq t}|\Delta X^n(s)|\Big\}\bigg)\notag\\
&\leq \max_{1\leq n\leq n_0}2\lfloor nt \rfloor E^n\bigg|\frac{\xi_1}{a(n)}\bigg|
+ 2ct^{1/\alpha}(E|S_\alpha| + 1)< \infty.
\end{align*}

Since $(X^n)$ satisfies UT, we apply Theorem 34 in \cite{KP1995} to conclude that if $(H^n,X^n) \Rightarrow (H,A)$ in $(\mathbb{D}_{\R^2},J_1)$, then there exists an appropriate filtration such that as $n\rightarrow \infty$,
\begin{equation*}
\left\{\int_0^t H^n(s-) \text{ }dX^n(s)\right\}_{t\geq 0} \Longrightarrow \left\{\int_0^t H(s-) \text{ }dA(s)\right\}_{t\geq 0} \text{ in } (\mathbb{D}_{\R},J_1).
\end{equation*}

%\vspace{-.25in}
%\qed
\end{proof}

By an easy application of this theorem, we see that the stochastic integral of a scaled simple random walk with respect to itself converges in distribution to the stochastic integral of Brownian motion with respect to Brownian motion.

\begin{example}
Let $\displaystyle X^n(t) = \sum\limits_{i=1}^{\lfloor nt \rfloor}\frac{\xi_i}{\sqrt{n}}$, where $\xi_i$ are i.i.d. mean 0 random variables with variance 1.  By Donsker's Theorem, $\{X^n(t)\}_{t\geq 0} \Rightarrow \{B(t)\}_{t\geq 0}$ in $(\mathbb{D}_\R,J_1)$, so $(X^n,X^n)\Rightarrow (B,B)$ in $(\mathbb{D}_{\R^2},J_1)$.  Then by Theorem \ref{thm1} as $n\rightarrow \infty$,
\begin{equation}
\left\{\int_0^t X^n(s-) \text{ }dX^n(s)\right\}_{t\geq 0} \Longrightarrow \left\{\int_0^t B(s) \text{ }dB(s)\right\}_{t\geq 0} \text{ in } (\mathbb{D}_{\R},J_1).
\end{equation}
\end{example}

We now establish a result concerning the weak convergence of stochastic integrals driven by CTRWs to a stochastic integral driven by the CTRW limit process $B(E_t)$ when the CTRWs have infinite mean waiting times and jumps in the domain of normal attraction of a normal law.  Specifically:

\begin{theorem}
Let $(\xi_i)_{i\in \N}$ be i.i.d. mean 0 random variables with $E|\xi_1|^2 = c^2$, so $\xi_1$ belongs to the domain of normal attraction of a normal random variable $Z=N(0,1)$.  Let $(\tau_i)_{i\in \N}$ be i.i.d. strictly $\beta$-stable random variables, $\beta \in (0,1)$.  Define
\begin{equation}
X^n(t) = \sum_{i=1}^{N_{nt}} \frac{\xi_i}{cn^{\beta/2}}.
\label{eqn:thm2CTRW}
\end{equation}
If $(H^n,X^n) \Rightarrow (H,B(E))$ in $(\mathbb{D}_{\R^2},J_1)$, then there exists a filtration $(\mathcal{F}_{t})$ such that $B(E)$ is an $(\mathcal{F}_t)$-semimartingale, $H$ is a $(\mathcal{F}_t)$-adapted \cad\ process, and as $n\rightarrow \infty$
\begin{equation}
\left\{\int_0^t H^n(s-) \text{ }dX^n(s)\right\}_{t\geq 0} \Longrightarrow \left\{\int_0^t H(s-) \text{ }dB(E_s)\right\}_{t\geq 0} \text{ in } (\mathbb{D}_\R,J_1).
\end{equation}
\label{thm2}
\end{theorem}
As in the proof of Theorem \ref{thm1}, we prove this result by first verifying that the conditions of Proposition \ref{prop:UT} hold and then applying Theorem 34 in \cite{KP1995}.  However, since this case involves random waiting times with infinite mean, our proof also requires a uniform bound on the expectation of the scaled counting process $N_{nt}/n^\beta$.  We prove the necessary lemma here before presenting the proof of Theorem \ref{thm2}.

\begin{lemma}
Let $(\tau_i)_{i\in \N}$ be i.i.d. strictly $\beta$-stable random variables, $\beta \in (0,1)$.  Then for each $t\geq 0$,
\begin{equation}
E\left[\frac{N_{nt}}{n^\beta}\right] \leq t^{\beta}M,
\end{equation}
for all $n\geq 1$, where $M$ is a constant.
\label{lem:unifbound}
\end{lemma}
\begin{proof}
The key to bounding the expectation will be the asymptotics of the density of $\tau_1$ at zero.  To get at $\tau_1$, we write the expectation as the integral of the tail probabilities of the process $N_{nt}/n^{\beta}$.  Using the inverse relationship given by Equation (\ref{eqn:relationship}) between $N_t$ and $T(n)$ as well as the fact that the $\tau_i$ are strictly $\beta$-stable, we obtain
\begin{align*}
E\left[\frac{N_{nt}}{n^\beta}\right] &= \int_0^\infty P(N_{nt}>xn^\beta )\text{ }dx\\
&=\int_0^\infty P\left(T\left(\lceil xn^\beta \rceil\right) < nt\right)\text{ }dx \\
&=\int_0^\infty P\left(\tau_1< \frac{nt}{\left(\lceil xn^\beta \rceil \right)^{1/\beta}}\right)\text{ }dx.
\end{align*}

Since $nx^{1/\beta}\leq \lceil xn^\beta\rceil^{1/\beta}$, the expectation is bounded above by $\int_0^\infty P\left(\tau_1 < \frac{t}{x^{1/\beta}}\right)\text{ }dx$.  By change of variables, this integral can be re-written in a more useful form.  Explicitly:

\begin{equation*}
\int_0^\infty P\left(\tau_1 < \frac{t}{x^{1/\beta}}\right) \text{ }dx= \beta t^{\beta} \int_0^\infty \frac{P(\tau_1<y)}{y^{1+\beta}}\text{ }dy.
\end{equation*}

It is not difficult to bound this integral away from 0, so it will remain to show that the integral is also bounded near 0.  Observe that for any $y_0>0$,
\begin{align*}
E\left[\frac{N_{nt}}{n^\beta}\right]&\leq \beta t^{\beta}\left[\int_0^{y_0} \frac{P(\tau_1<y)}{y^{1+\beta}}\text{ }dy + \int_{y_0}^\infty \frac{1}{y^{1+\beta}}\text{ }dy \right]\\
&\leq t^{\beta} \Bigg\{\beta \bigg[y_0\cdot \sup_{0<y\leq y_0}\frac{P(\tau_1<y)}{y^{1+\beta}} + \frac{1}{\beta y_0^\beta}  \bigg]\Bigg\}.
\end{align*}

The objective is now to establish that the term in braces can be bounded by a finite number.  By Theorem 2.5.2 in \cite{Zolotarev}, $\displaystyle \lim_{y\rightarrow 0} \frac{f_{\tau_1}(y)}{Z(y)}=1$, where $f_{\tau_1}(y)$ is the density of $\tau_1$ and
\begin{equation*}
Z(y) = A \frac{e^{-C y^{\frac{-\beta}{1-\beta}}}}{y^{\frac{2-\beta}{2(1-\beta)}}}, \text{ where } A=\frac{\beta^{\frac{2-\beta}{2(1-\beta)}}}{\sqrt{2\pi\beta(1-\beta)}} \text{ and }
C= (1-\beta)\beta^{\frac{\beta}{1-\beta}}>0.
\end{equation*}
To apply this result, we note that the derivative of the cumulative distribution function is the density, and apply L'Hospital's rule to obtain:
\begin{align*}
\lim_{y\rightarrow 0^+} \frac{P(\tau_1 <y)}{y^{1+\beta}}&=
\lim_{y\rightarrow 0^+} \frac{f_{\tau_1}(y)}{(1+\beta)y^\beta}
=\lim_{y\rightarrow 0^+} \frac{f_{\tau_1}(y)}{Z(y)} \cdot \lim_{y\rightarrow 0^+} \frac{Z(y)}{(1+\beta)y^\beta}\\
&= \frac{A}{1+\beta}\lim_{y\rightarrow 0^+} \frac{e^{-C y^{\frac{-\beta}{1-\beta}}}}
{y^{\frac{2-\beta}{2(1-\beta)}+\beta}}
= 0.
\end{align*}

Therefore, there exists $0<\delta<y_0$ such that for all $0<y<\delta$, $\frac{P(\tau_1<y)}{y^{1+\beta}} \leq 1$ and
\begin{align*}
\sup_{0<y\leq y_0}\frac{P(\tau_1<y)}{y^{1+\beta}}
&=\max\left(\sup_{0<y<\delta}\frac{P(\tau_1<y)}{y^{1+\beta}}, \max_{\delta\leq y \leq y_0}\frac{P(\tau_1<y)}{y^{1+\beta}}\right)\\
&\leq \max\left(1,\max_{\delta\leq y \leq y_0}\frac{1}{y^{1+\beta}}\right)
\leq 1 + \frac{1}{\delta^{1+\beta}} < \infty.
\end{align*}

Let $M = \beta\Big[y_0\Big(1+\frac{1}{\delta^{1+\beta}}\Big) + \frac{1}{\beta y_0^\beta}\Big]$.  Then for all $n$, $\displaystyle E\bigg[\frac{N_{nt}}{n^\beta}\bigg] \leq t^{\beta}M$, as desired.
%\qed
\end{proof}

\begin{proof}[Proof of Theorem \ref{thm2}]
We first show that $(X^n(t))_{n\geq 1}$ is a sequence of martingales.  Note that the jumps $\xi_i$ have mean 0 and finite first moment and for each $n$, $E[N_{nt}]<\infty$ for all $t\geq 0$.  By the same techniques used in Lemma \ref{lem:martingale}, we see that $X^n(t)$ is an $\mathcal{F}_t^n$-martingale (and therefore a local martingale) for each $n$, where $\mathcal{F}_t^n$ is the natural filtration of $X^n$.  Therefore, the first hypothesis of Proposition \ref{prop:UT} holds.

Now we verify the second condition: $\sup_n E^n[\sup_{s\leq t}|\Delta X^n(s)|]<\infty$ for each $t<\infty$.  Fix $t < \infty$.  Observe that $X^n(s)$ is a c\`{a}dl\`{a}g step process with jumps of size $\left|\frac{\xi_k}{cn^{\beta/2}}\right|$ at times $\frac{T(k)}{n}$, $k\in \N$.

As in Theorem \ref{thm1}, the goal is to obtain for each $t$ a uniform bound in $n$ on the expectation.  To this end, we first expand the expectation using the definition and standard technique of computing the integral of the tail probabilities.  To apply L\'{e}vy's Inequality, we symmetrize by replacing $\xi_k$ with $\epsilon_k\xi_k$ where $(\epsilon_k)$ is a sequence of i.i.d. Rademacher random variables, independent of $(\xi_k)$.  Then $|\xi_k|=|\epsilon_k\xi_k|$ and $\epsilon_k\xi_k$ are i.i.d. in $DOA(Z)$.  Let $\widetilde{S}(m)$ denote the symmetrized sum.  Since $X^n(s)$ has only finitely many jumps by time $t$, we can replace the $\sup$ up to time $t$ by the $\max$ up to $N_{nt}$ and the following holds:
\begin{align*}
E^n[\sup_{s\leq t}|\Delta X^n(s)| ]
&= \sum_{m=0}^\infty \left[\int_0^\infty P^n\left(\max_{1\leq k\leq m}\frac{|\epsilon_k\xi_k|}{cn^{\beta/2}}>x \right)\text{ }dx \right] P(N_{nt}=m)\\
&\leq \sum_{m=0}^\infty \left[\int_0^\infty 2P^n\left(\left|\sum_{k=1}^m \frac{\epsilon_k\xi_k}{cn^{\beta/2}}\right|>x \right)\text{ }dx \right]P(N_{nt}=m)\\
&= 2\sum_{m=0}^\infty E^n\left|\frac{\widetilde{S}(m)}{cn^{\beta/2}} \right| P(N_{nt}=m).
\end{align*}

We can easily bound the sum for a finite number of terms.  The goal is then to show that the tail of the sum is bounded.  Since $\epsilon_k\xi_k \in DOA(Z)$, then by \cite[Exercise 9, p. 91]{AG}, $\displaystyle{E\left|\frac{\widetilde{S}(m)}{m^{1/2}}\right| \rightarrow E|Z|}$, so there exists $m_0$ such that for all $m>m_0$, $\displaystyle{E\left|\frac{\widetilde{S}(m)}{m^{1/2}}\right|\leq E|Z| + 1}$.  We break up the sum at $m_0$:
\begin{equation}
\frac{2}{cn^{\beta/2}}\left[\sum_{m=0}^{m_0} E^n|\tilde{S}(m)| P(N_{nt}=m)
+ \sum_{m=m_0+1}^\infty E^n\left|\frac{\widetilde{S}(m)}{m^{1/2}} \right| m^{1/2}P(N_{nt}=m)\right].
\label{eqn:breakupsum}
\end{equation}

The sum up to $m_0$ is finite since $E^n|\widetilde{S}(m)|\leq E|\xi_1|<\infty$ by the triangle inequality.  By applying the bound given above on the first moment of $\widetilde{S}(m)/m^{1/2}$, we can bound the tail of the sum by:
\begin{equation}
\frac{2}{c}(E|Z|+1)E\bigg[\left(\frac{N_{nt}}{n^\beta}\right)^{1/2}\bigg].
\label{eqn:tail}
\end{equation}

\noindent Finally, we show Equation (\ref{eqn:tail}) is finite.  Since $\sqrt{x} \leq x+1$ for all $x\geq 0$, and $E[N_{nt}/n^\beta]\leq t^\beta M$ by Lemma \ref{lem:unifbound}, we have
\begin{equation}
E\bigg[\left(\frac{N_{nt}}{n^\beta}\right)^{1/2}\bigg] \leq t^\beta M + 1 < \infty.
\end{equation}

We have shown that Equation (\ref{eqn:breakupsum}) is finite and is an upper bound for $E^n[\sup_{s\leq t}|\Delta X^n(s)| ]$.  Therefore, the second hypothesis of Proposition \ref{prop:UT} is verified.

Since $(X^n)$ satisfies UT, an application of Theorem 34 in \cite{KP1995} yields the conclusion that if $(H^n,X^n) \Rightarrow (H,B(E))$ in $(\mathbb{D}_{\R^2},J_1)$ then there exists an appropriate filtration such that as $n\rightarrow \infty$,
\begin{equation*}
\left\{\int_0^t H^n(s-) \text{ }dX^n(s)\right\}_{t\geq 0} \Longrightarrow \left\{\int_0^t H(s-) \text{ }dB(E_s)\right\}_{t\geq 0} \text{ in } (\mathbb{D}_\R,J_1).
\end{equation*}

%\vspace{-.25in}
%\qed
\end{proof}

The following example proves the conjecture of Germano \textit{et al.} \cite{Scalas2010}: that the stochastic integral of the scaled compound Poisson process with respect to the scaled compound Poisson process converges weakly to the stochastic integral of Brownian motion with respect to Brownian motion.

\begin{example}
\label{ex:Germano_conj}
Let $(\xi_i)_{i\in \N}$ be i.i.d. normal random variables with mean 0 and variance 1, and $(\tau_i)_{i\in \N}$ be i.i.d. exponential random variables with $\lambda = 1$.  Then $X(t) = \sum_{i=1}^{N_t}\xi_i$ is the usual compound Poisson process, and $E(N_t) = t$.  Let $X^n(t) = \sum_{i=1}^{N(nt)}\frac{\xi_i}{\sqrt{n}}$ be the scaled version.  Then $X^n(t)$ is a martingale for each $n$.  Fix $t<\infty$.  We show $E^n[\sup_{s\leq t}|\Delta X^n(s)| ] < \infty$ using the same techniques as in the proof of Theorem \ref{thm2}.  The sum in Equation (\ref{eqn:breakupsum}) (with $\beta = 1$) up to $m_0$ can be bounded by a finite number and for $m>m_0$, the tail is bounded by
\begin{equation*}
(E|Z|+1)E\bigg[\Big(\frac{N_{nt}}{n}\Big)^{1/2}\bigg] \leq (E|Z|+1)\sqrt{t} < \infty.
\end{equation*}
Since $(X^n,X^n) \Rightarrow (B,B)$ in $(\mathbb{D}_\R, J_1)$ and the above shows the sequence $(X^n)$ is good, then as $n\rightarrow \infty$,
\begin{equation}
\left\{\int_0^t X^n(s-) \text{ }dX^n(s)\right\}_{t\geq 0} \Longrightarrow \left\{\int_0^t B(s) \text{ }dB(s)\right\}_{t\geq 0} \text{ in } (\mathbb{D}_{\R},J_1).
\end{equation}
\end{example}

\begin{rem}Meerschaert and Sheffler prove a general functional central limit theorem for CTRWs, in which the CTRW limit process is a time-changed L\'{e}vy process $A(E_t)$ (see Theorem 4.2 in \cite{MS:2004}).  This theorem holds in the $M_1$-topology but not in the stronger $J_1$-topology, except for the special case where $A$ is Brownian motion (as discussed in Section \ref{sec:background}).  In proving Theorems \ref{thm1} and \ref{thm2} we rely on Kurtz and Protter's weak convergence theorem for stochastic integrals which holds in the $J_1$-topology.  While we hope that a version of this theorem exists in the $M_1$-topology, this question remains open.
\end{rem}

\section{Applications}
\label{sec:App}
Knowing that a sequence of scaled CTRWs is good is useful in many situations.  By definition, goodness of a sequence of scaled CTRWs $(X^n)$ implies that if $(H^n,X^n)$ converges jointly to $(H,X)$, then the integral of $H^n$ driven by $X^n$ converges weakly to the integral of $H$ driven by $X$.  Since the stochastic integral driven by a scaled CTRW can be defined as a sum, the stochastic integral driven by a scaled CTRW can be easily simulated and used as an approximation for the stochastic integral driven by its limit process.  Additionally, goodness is a necessary condition on the driving process for the weak convergence of stochastic differential equations to hold.  Section \ref{subsec:approx} focuses on two examples which apply the result of Theorem \ref{thm2}, in which the CTRW limit process is a time-changed Brownian motion, to cases of weak convergence of stochastic differential equations driven by scaled CTRWs.  Section \ref{subsec:particle} describes the connection with particle tracking.

\subsection{Weak convergence of stochastic differential equations}
\label{subsec:approx}
The results in Section \ref{sec:Results} can be combined with a theorem by Kurtz and Protter on weak convergence of solutions to stochastic differential equations \cite[Theorem 38]{KP1995}.  A key condition in this theorem is that the driving process is good.  The conditions of the theorem are simplified in the case that the driving process is a scaled CTRW known to be good.  We state an example in the case that the driving process satisfies the conditions in Theorem \ref{thm2}, where the limit process is a time-changed Brownian motion.
\begin{example}
Let $(X^n)_{n\geq 1}$ be a sequence of semimartingales and $(Z^n)_{n\geq 1}$ be a sequence of CTRWs satisfying the conditions in Theorem \ref{thm2}.  Assume the following:
\begin{enumerate}[$\bullet$]
\item $(X^n,Z^n)$ satisfies
\begin{equation}
X^n_t = \int_0^t F^n(X^n)_{s-}\text{ } dZ^n_s.
\label{inteqn}
\end{equation}
\item $(X^n,Z^n)$ is relatively compact in $(\mathbb{D}_{\R^2},J_1)$.
\item $F^n,F$ are $J_1$-continuous.
\end{enumerate}
Then since the sequence $(Z^n)_{n\geq 1}$ is good by Theorem \ref{thm2} and $Z^n \Rightarrow B(E)$ by Theorem \ref{thm:B(E)}, it follows from Theorem 38 in \cite{KP1995} that any limit point of the sequence $(X^n)$ satisfies
\begin{equation}
X_t = \int_0^t F(X)_{s-}\text{ } dB_{E(s)}.
\end{equation}
\label{thm:CTRWSDE}
\end{example}

The next example involves a special kind of stochastic integral equation for which the form of the solution is known.  Consider $X_t = 1 + \int_0^t X_{s-}\text{ }dY_s$, where $Y$ is a semimartingale and $Y_0=0$.
By \cite[Chapter II, Theorem 37]{Protter}, the solution $X$ is a semimartingale called the \emph{stochastic exponential} of $Y$ and is given by:
\begin{equation}
X_t = \exp\left\{Y_t - \frac{1}{2}[Y,Y]_t\right\}\prod_{0<s\leq t} (1+\Delta Y_s)\exp\left\{-\Delta Y_s + \frac{1}{2}(\Delta Y_s)^2\right\},
\end{equation}
where the infinite product converges.  Using Theorem \ref{thm2} and consequences of the goodness of the scaled CTRWs, we are able to verify directly, without checking any conditions on the stochastic integral equations, that the stochastic exponential of these scaled CTRWs converges weakly to the stochastic exponential of the CTRW limit process $B(E_t)$.

\begin{example}
Let $\displaystyle Z^n(t) = \sum\limits_{i=1}^{N(nt)} \frac{\xi_i}{n^{\beta/2}}$ where $(\xi_i)$ are i.i.d. Rademacher random variables and $(\tau_i)$ are i.i.d. strictly $\beta$-stable random variables, $\beta \in (0,1)$.  Then $Z^n \Rightarrow B(E)$ by Theorem \ref{thm:B(E)} and $(Z^n)$ is good by Theorem \ref{thm2}.  Consider
\begin{equation}
X^n_t = 1 + \int_0^t X_{s-}^n \text{ } dZ_s^n.
\label{mySDE}
\end{equation}

\noindent We want to verify that $X^n \Rightarrow X$ in $(\mathbb{D}_\R,J_1)$, where $X$ is a solution of the limiting equation
\begin{equation}
X_t = 1 + \int_0^t X_{s-}\text{ }dB_{E_s}.
\label{BEIE}
\end{equation}

We first analyze the form of the solutions.  Since $B(E_t)$ is a semimartingale, the stochastic differential equation (\ref{BEIE}) has a unique solution which is given by $X_t = \exp\{B_{E_t} - \frac{E_t}{2}\}$ \cite{Kei}.  Since $Z^n$ is a martingale, hence a semimartingale, with $Z^n(0)=0$, the stochastic exponential of $Z^n$ is the (unique) solution to (\ref{mySDE}) and is given by

\begin{align*}
X^n(t) &= \exp\left\{Z^n(t) - \frac{1}{2}[Z^n,Z^n]_t\right\}\prod_{0<s\leq t}(1+\Delta Z^n(s))\exp\left\{-\Delta Z^n(s) + \frac{1}{2}(\Delta Z^n(s))^2\right\}\\
&\equiv W^n(t) \cdot A^n(t).
\end{align*}

To handle $\displaystyle W^n(t) = \exp\left\{Z^n(t) - \frac{1}{2}[Z^n,Z^n]_t\right\}$,
we appeal to the following two results; the first is an immediate corollary to Theorem 36 in \cite{KP1995} and the second is Theorem 10.17 in \cite{Jacod}:
\begin{enumerate}[1.]
\setlength{\itemsep}{0.1in}
\item If $(Z^n)$ is good and $Z^n\Rightarrow Z$ in $(\mathbb{D}_\R,J_1)$, then $\displaystyle(Z^n,[Z^n,Z^n])\Rightarrow (Z,[Z,Z]) \text{ in } (\mathbb{D}_{\R^2},J_1)$.
\label{thm:quadvarconv}
\item Let $Y$ be an $(\mathcal{F}_t)$-semimartingale and $(T_t)$ be a finite continuous $(\mathcal{F}_t)$-time-change.  Then $\displaystyle [Y\circ T,Y\circ T] = [Y,Y]\circ T$.
\label{thm:Jacod}
\end{enumerate}

Since Brownian motion $B$ is a semimartingale, $E$ is a continuous nondecreasing process with $E_t<\infty$ almost surely, and the quadratic variation $[B,B]_t = t$, then $[B(E),B(E)]_t = ([B,B]\circ E)_t = E_t$.

Since $Z^n \Rightarrow B(E)$ in $(\mathbb{D}_\R,J_1)$, $[Z^n,Z^n]\Rightarrow[B(E),B(E)]=E$ in $(\mathbb{D}_\R,J_1)$.  Now using the $J_1$-continuity of addition on $(Z^n, \frac{1}{2}[Z^n,Z^n])$ and the continuous-mapping theorem, we have:
\begin{equation*}
Z^n -  \frac{1}{2}[Z^n,Z^n] \Rightarrow B(E) -\frac{1}{2}E \text{ in } (\mathbb{D}_\R,J_1).
\end{equation*}
By the $J_1$-continuity of the exponential and the continuous-mapping theorem,
\begin{equation*}
\exp\left\{Z^n - \frac{1}{2}[Z^n,Z^n]\right\} \Rightarrow \exp\left\{B(E) - \frac{1}{2}E\right\} \text{ in } (\mathbb{D}_\R,J_1).
\end{equation*}
So $W^n \Rightarrow X$ in $(\mathbb{D}_\R,J_1)$.  We need to show $W^n \cdot A^n \Rightarrow X$ in $(\mathbb{D}_\R,J_1)$.  If $A^n \Rightarrow 1$, then $A^n \rightarrow 1$ in probability and so $(A^n,W^n)\Rightarrow (1,X)$ in $\mathbb{D}\times \mathbb{D}$.  Since 1 and $X$ are both continuous, multiplication here is continuous by $J_1$-continuity of multiplication \cite{Whitt}.  By the continuous-mapping theorem, $W^n\cdot A^n \Rightarrow X$.  It remains to show $A^n \Rightarrow 1$, or equivalently $\log A^n \Rightarrow 0$.

Observe that
\begin{equation*}
\log A^n(t) = \sum_{i=1}^{N(nt)} [\log(1+\Delta Z^n(s)) - \Delta Z^n(s) + \frac{1}{2}(\Delta Z^n(s))^2].
\end{equation*}
Also note that $|\Delta Z^n(t)| = |\frac{\xi_k}{n^{\beta/2}}|$ for $t = \frac{T(k)}{n}$ and 0 otherwise, since $Z^n(t)$ is a c\`{a}dl\`{a}g step process with jumps at times $\frac{T(k)}{n}$, $k\in \N$.  For $n>2^{2/\beta}$, $\displaystyle |\Delta Z^n(s)| = \bigg|\frac{\xi_i}{n^{\beta/2}}\bigg| \leq \frac{1}{n^{\beta/2}}<\frac{1}{2}$, so
\begin{equation*}
|\log(1+\Delta Z^n(s)) - \Delta Z^n(s) + \frac{1}{2}(\Delta Z^n(s))^2| \leq \frac{|\Delta Z^n_s|^3}{3},
\end{equation*}
by the alternating series estimation theorem.
Therefore,
\begin{equation*}
\sum_{i=1}^{N(nt)} \bigg|\log(1+\Delta Z^n(s)) - \Delta Z^n(s) + \frac{1}{2}(\Delta Z^n(s))^2\bigg|
\leq \sum_{i=1}^{N(nt)} \bigg|\frac{\xi_i^3}{n^{3\beta/2}}\bigg|
\leq \sum_{i=1}^{N(nt)} \frac{1}{n^{3\beta/2}} = \frac{N(nt)}{n^\beta}\frac{1}{n^{\beta/2}}.
\end{equation*}

\noindent Since $\displaystyle \lim\limits_{n\rightarrow \infty} \frac{1}{n^{\beta/2}} = 0$ and $\displaystyle \frac{N(nt)}{n^\beta}\Rightarrow E(t)$ in $(\mathbb{D}_\R,J_1)$ as $n\rightarrow \infty$ by Corollary 3.4 in \cite{MS:2004}, $\displaystyle \frac{N(nt)}{n^\beta}\frac{1}{n^{\beta/2}} \Rightarrow 0$ in $(\mathbb{D}_\R,J_1)$, which implies $\log A^n(t) \Rightarrow 0$, thereby confirming that $X^n \Rightarrow X$.
\end{example}

\subsection{Particle tracking}
\label{subsec:particle}
These results also have applications to particle tracking.  Particle tracking is a method of solving partial differential equations in cases where an analytic solution (i.e., closed form solution) cannot be found.  Particle tracking involves first finding the stochastic differential equation (SDE) corresponding to the forward Kolmogorov equation and second simulating the solution to the SDE \cite{Zhang}.  For example, a special case of Theorem 4.1 in \cite{HKU} considers the SDE:
\begin{equation}
dX_t = a(X_t)dE_t + b(X_t)dB_{E_t},X_0=x,
\label{BESDE}
\end{equation}
and shows that the density $p^X(t,x,y)$ of $X_t$ satisfies in the weak sense the following forward Kolmogorov equation:
\begin{align}
\label{TFPDE}
D_*^{\beta}&p^X(t,x,y) = \left[-\frac{\partial}{\partial y}a(y) + \frac{1}{2}\frac{\partial^2}{\partial y^2}b^2(y) \right]p^X(t,x,y)\\
&p^X(0,x,y) = \delta_x(y)\notag,
\end{align}
where $D_*^\beta$ is the Caputo fractional derivative defined as
\begin{equation}
D_*^\beta g(t) = \frac{1}{\Gamma(1-\beta)} \int_0^t \frac{g'(u)}{(t-u)^\beta}\text{ }du.
\end{equation}
Because of this correspondence between the SDE (\ref{BESDE}) and the time-fractional differential equation (\ref{TFPDE}), particle tracking can be used to solve (\ref{TFPDE}).  The solution to the SDE can be simulated using CTRWs where the waiting times arise from $\beta$-stable random variables as in Theorem \ref{thm2}.

\bibliographystyle{plain}
\bibliography{CTRW&StochInt1}   % name your BibTeX data base

\end{document}